\documentclass[12pt]{article}
\usepackage{latexsym, amssymb, amsthm}
\usepackage{dsfont}

\DeclareMathAlphabet\gothic{U}{euf}{m}{n}

\makeatletter
\def\eqnarray{\stepcounter{equation}\let\@currentlabel=\theequation
\global\@eqnswtrue
\tabskip\@centering\let\\=\@eqncr
$$\halign to \displaywidth\bgroup\hfil\global\@eqcnt\z@
  $\displaystyle\tabskip\z@{##}$&\global\@eqcnt\@ne
  \hfil$\displaystyle{{}##{}}$\hfil
  &\global\@eqcnt\tw@ $\displaystyle{##}$\hfil
  \tabskip\@centering&\llap{##}\tabskip\z@\cr}

\def\endeqnarray{\@@eqncr\egroup
      \global\advance\c@equation\m@ne$$\global\@ignoretrue}

\def\@yeqncr{\@ifnextchar [{\@xeqncr}{\@xeqncr[5pt]}}
\makeatother

\begin{document}
\bibliographystyle{tom}

\newtheorem{lemma}{Lemma}[section]
\newtheorem{thm}[lemma]{Theorem}
\newtheorem{cor}[lemma]{Corollary}
\newtheorem{prop}[lemma]{Proposition}

\theoremstyle{definition}

\newtheorem{remark}[lemma]{Remark}
\newtheorem{exam}[lemma]{Example}
\newtheorem{definition}[lemma]{Definition}

\newcommand{\gota}{\gothic{a}}
\newcommand{\gotb}{\gothic{b}}
\newcommand{\gotc}{\gothic{c}}
\newcommand{\gotd}{\gothic{d}}
\newcommand{\gote}{\gothic{e}}
\newcommand{\gotf}{\gothic{f}}
\newcommand{\gotg}{\gothic{g}}
\newcommand{\gothh}{\gothic{h}}
\newcommand{\gotk}{\gothic{k}}
\newcommand{\gotm}{\gothic{m}}
\newcommand{\gotn}{\gothic{n}}
\newcommand{\gotp}{\gothic{p}}
\newcommand{\gotq}{\gothic{q}}
\newcommand{\gotr}{\gothic{r}}
\newcommand{\gots}{\gothic{s}}
\newcommand{\gott}{\gothic{t}}
\newcommand{\gotu}{\gothic{u}}
\newcommand{\gotv}{\gothic{v}}
\newcommand{\gotw}{\gothic{w}}
\newcommand{\gotz}{\gothic{z}}
\newcommand{\gotA}{\gothic{A}}
\newcommand{\gotB}{\gothic{B}}
\newcommand{\gotG}{\gothic{G}}
\newcommand{\gotL}{\gothic{L}}
\newcommand{\gotS}{\gothic{S}}
\newcommand{\gotT}{\gothic{T}}

\newcounter{teller}
\renewcommand{\theteller}{(\alph{teller})}
\newenvironment{tabel}{\begin{list}%
{\rm  (\alph{teller})\hfill}{\usecounter{teller} \leftmargin=1.1cm
\labelwidth=1.1cm \labelsep=0cm \parsep=0cm}
                      }{\end{list}}

\newcounter{tellerr}
\renewcommand{\thetellerr}{(\roman{tellerr})}
\newenvironment{tabeleq}{\begin{list}%
{\rm  (\roman{tellerr})\hfill}{\usecounter{tellerr} \leftmargin=1.1cm
\labelwidth=1.1cm \labelsep=0cm \parsep=0cm}
                         }{\end{list}}

\newcounter{tellerrr}
\renewcommand{\thetellerrr}{(\Roman{tellerrr})}
\newenvironment{tabelR}{\begin{list}%
{\rm  (\Roman{tellerrr})\hfill}{\usecounter{tellerrr} \leftmargin=1.1cm
\labelwidth=1.1cm \labelsep=0cm \parsep=0cm}
                         }{\end{list}}

\newcounter{proofstep}
\newcommand{\nextstep}{\refstepcounter{proofstep}\vertspace \par 
          \noindent{\bf Step \theproofstep} \hspace{5pt}}
\newcommand{\firststep}{\setcounter{proofstep}{0}\nextstep}

\newcommand{\Ni}{\mathds{N}}
\newcommand{\Qi}{\mathds{Q}}
\newcommand{\Ri}{\mathds{R}}
\newcommand{\Ci}{\mathds{C}}
\newcommand{\Ti}{\mathds{T}}
\newcommand{\Zi}{\mathds{Z}}
\newcommand{\Fi}{\mathds{F}}
\newcommand{\Ki}{\mathds{K}}

\renewcommand{\proofname}{{\bf Proof}}

\newcommand{\simh}{{\stackrel{{\rm cap}}{\sim}}}
\newcommand{\ad}{{\mathop{\rm ad}}}
\newcommand{\Ad}{{\mathop{\rm Ad}}}
\newcommand{\alg}{{\mathop{\rm alg}}}
\newcommand{\clalg}{{\mathop{\overline{\rm alg}}}}
\newcommand{\Aut}{\mathop{\rm Aut}}
\newcommand{\arccot}{\mathop{\rm arccot}}
\newcommand{\capp}{{\mathop{\rm cap}}}
\newcommand{\rcapp}{{\mathop{\rm rcap}}}
\newcommand{\curl}{\mathop{\rm curl}}
\newcommand{\diam}{\mathop{\rm diam}}
\newcommand{\divv}{\mathop{\rm div}}
\newcommand{\dom}{\mathop{\rm dom}}
\newcommand{\codim}{\mathop{\rm codim}}
\newcommand{\RRe}{\mathop{\rm Re}}
\newcommand{\IIm}{\mathop{\rm Im}}
\newcommand{\tr}{{\mathop{\rm Tr \,}}}
\newcommand{\Tr}{{\mathop{\rm Tr \,}}}
\newcommand{\Vol}{{\mathop{\rm Vol}}}
\newcommand{\card}{{\mathop{\rm card}}}
\newcommand{\rank}{\mathop{\rm rank}}
\newcommand{\supp}{\mathop{\rm supp}}
\newcommand{\sgn}{\mathop{\rm sgn}}
\newcommand{\essinf}{\mathop{\rm ess\,inf}}
\newcommand{\esssup}{\mathop{\rm ess\,sup}}
\newcommand{\Int}{\mathop{\rm Int}}
\newcommand{\lcm}{\mathop{\rm lcm}}
\newcommand{\loc}{{\rm loc}}
\newcommand{\HS}{{\rm HS}}
\newcommand{\Trn}{{\rm Tr}}
\newcommand{\n}{{\rm N}}
\newcommand{\WOT}{{\rm WOT}}

\newcommand{\at}{@}

\newcommand{\spann}{\mathop{\rm span}}
\newcommand{\one}{\mathds{1}}

\hyphenation{groups}
\hyphenation{unitary}

\newcommand{\ca}{{\cal A}}
\newcommand{\cb}{{\cal B}}
\newcommand{\cc}{{\cal C}}
\newcommand{\cd}{{\cal D}}
\newcommand{\ce}{{\cal E}}
\newcommand{\cf}{{\cal F}}
\newcommand{\ch}{{\cal H}}
\newcommand{\chs}{{\cal HS}}
\newcommand{\ci}{{\cal I}}
\newcommand{\ck}{{\cal K}}
\newcommand{\cl}{{\cal L}}
\newcommand{\cm}{{\cal M}}
\newcommand{\cn}{{\cal N}}
\newcommand{\co}{{\cal O}}
\newcommand{\cp}{{\cal P}}
\newcommand{\cs}{{\cal S}}
\newcommand{\ct}{{\cal T}}
\newcommand{\cx}{{\cal X}}
\newcommand{\cy}{{\cal Y}}
\newcommand{\cz}{{\cal Z}}

\vspace*{1cm}
\begin{center}
{\large\bf A Friedlander type estimate for Stokes operators} \\[5mm]
\large  C. Denis and A.F.M. ter Elst

\end{center}

\vspace{5mm}

\begin{list}{}{\leftmargin=1.8cm \rightmargin=1.8cm \listparindent=10mm 
   \parsep=0pt}
\item
\small
{\sc Abstract}.
Let $\Omega \subset \Ri^d$ be a bounded open connected set with Lipschitz boundary.
Let $A^N$ and $A^D$ be the Neumann Stokes operator and Dirichlet Stokes operator on $\Omega$,
respectively.
Further let $\lambda_1^N \leq \lambda_2^N \leq \ldots$ and 
$\lambda_1^D \leq \lambda_2^D \leq \ldots$ be the eigenvalues of $A^N$ and~$A^D$
repeated with multiplicity, respectively.
Then 
\[
\lambda_{n+1}^N < \lambda_n^D
\]
for all $n \in \Ni$.
\end{list}

\let\thefootnote\relax\footnotetext{
\begin{tabular}{@{}l}
{\em Mathematics Subject Classification}. 47A75, 35P15.\\
{\em Keywords}. Eigenvalues, counting function, Stokes operator.
\end{tabular}}

\section{Introduction} \label{Sfried1}

If $\lambda_1^D\leq \lambda_2^D\leq \ldots $ are the eigenvalues of the 
Dirichlet Laplacian and $\lambda_1^N \leq \lambda_2^N \leq \ldots$ the eigenvalues 
of the Neumann Laplacian on a domain $\Omega$ with Lipschitz boundary, then 
\begin{equation}
\lambda_{n+1}^N<\lambda_n^D
\label{eSfried1;20}
\end{equation}
for all $n\in \Ni$.
This result has a long history, starting with Payne \cite{Payne} and P\'olya \cite{Polya}.
In 1991, Friedlander \cite{Friedlander} proved that $\lambda_{n+1}^N\leq \lambda_n^D$ 
for a $C^1$ domain, using the Dirichlet-to-Neumann operator.
The method was modified and extended by Filonov \cite{Fil2} to prove a strict 
inequality on more general domains. 
A different and independent proof was obtained by Arendt--Mazzeo \cite{ArM2}.
On the other hand, Mazzeo \cite{Mazzeo} gave a counter example for manifolds.
Under certain geometric conditions on three dimensional domains
Hansson \cite{Hansson} proved (\ref{eSfried1;20}) for Dirichlet and Neumann eigenvalues
of the Heisenberg Laplacian.
This was extended to general Heisenberg groups and fairly general domains
by Frank--Laptev \cite{FrankLaptev}.
Other results are for powers $(-\Delta)^m$ of the Laplacian by 
Provenzano \cite{Provenzano} and a comparison between Robin and Dirichlet eigenvalues
by Gesztesy--Mitrea \cite{GesM} or for certain operators with mixed boundary conditions
by Lotoreichik--Rohleder \cite{LotRoh}.
In \cite{BRS} Behrndt--Rohleder--Stadler considered external domains and added a 
potential to the Laplacian.

The aim of this paper is to prove a similar result for the Stokes operator.
Let $\Omega \subset \Ri^d$ be an open bounded connected set.
One can define two versions of the Stokes operator: a Dirichlet version and 
a Neumann version.
The Dirichlet Stokes operator $A^D$ is defined in the closure 
$L_{2,\sigma}(\Omega)$ of solenoidal test functions
$ \{ u \in C_c^\infty(\Omega,\Ci^d) : \divv u = 0 \} $
in $L_2(\Omega,\Ci^d)$.
If $u,f \in L_{2,\sigma}(\Omega)$, then $u \in D(A^D)$ and $A^D u = f$
if and only if $u \in H^1_0(\Omega,\Ci^d)$, $\divv u = 0$ and 
there exists a $\pi \in L_2(\Omega)$ such that 
$f = - \Delta u + \nabla \pi$ in $H^{-1}(\Omega,\Ci^d)$.
Then $A^D$ is a positive self-adjoint operator in $L_{2,\sigma}(\Omega)$
with discrete spectrum.
Let $\lambda_1^D \leq \lambda_2^D \leq \ldots$ be the eigenvalues of $A^D$
repeated with multiplicity.

The Neumann Stokes operator $A^N$ is defined in the Hilbert space
$ \{ u \in L_2(\Omega,\Ci^d) : \divv u = 0 \} $.
In Proposition~\ref{pfried303} we shall give a precise definition of the 
self-adjoint operator~$A^N$.
Loosely speaking, $u \in D(A^N)$ and $A^N u = f$ if 
$u \in H^1(\Omega,\Ci^d)$, $\divv u = 0$ and there exists a
$\pi \in L_2(\Omega)$ such that $f = - \Delta u + \nabla \pi$ in $H^{-1}(\Omega,\Ci^d)$
and $\partial_\nu u = (\Tr \pi) \nu$, where $\partial_\nu$ is the 
outward normal derivative.
Also $A^N$ is a positive self-adjoint operator with discrete spectrum.
Let $\lambda_1^N \leq \lambda_2^N \leq \ldots$ be the eigenvalues of $A^D$
repeated with multiplicity.
The main result of this paper is that 
\begin{equation}
\lambda_{n+1}^N<\lambda_n^D
\label{eSfried1;1}
\end{equation}
for all $n \in \Ni$.

Although the Dirichlet and Neumann Laplacians are self-adjoint operators in 
the same Hilbert space $L_2(\Omega)$, this is not the case for the 
Stokes Dirichlet operator $A^D$ and Stokes Neumann operator~$A^N$.
The Stokes Dirichlet operator $A^D$ is a self-adjoint operator 
in the orthogonal complement in $L_2(\Omega,\Ci^d)$ of the space
\[
\{ \nabla \pi : \pi \in H^1(\Omega) \}
,  \]
whilst the Stokes Neumann operator~$A^N$ is a self-adjoint operator 
in the orthogonal complement in $L_2(\Omega,\Ci^d)$ of the space
\[
\{ \nabla \pi : \pi \in H^1_0(\Omega) \}
.  \]
Nevertheless, we relate the eigenvalues in~(\ref{eSfried1;1}).

As in Friedlander's original paper, the main ingredient for 
our proof is the Dirichlet-to-Neumann operator,
adapted to the Stokes problem.
We use sesquilinear form methods as in \cite{ArM} and \cite{ArM2}.
In Section~\ref{Sfried2} we prove an abstract theorem, Theorem~\ref{tstokes501},
on eigenvalue comparison of Dirichlet-type and Neumann-type operators
associated with forms, which relies (and expands) on a combination of the papers 
\cite{ArM}, \cite{ArM2}, \cite{AEKS} and the BSc(Honours) thesis \cite{Gor1}.
As an immediate application we obtain in Theorem~\ref{tfried202}
the classical result for the Laplacian on a Lipschitz domain.
Section~\ref{Sfried3} begins with the description of the Dirichlet and 
Neumann versions of the Stokes operator on a Lipschitz domain, as well 
as of the related Dirichlet-to-Neumann operator.
We then apply our abstract eigenvalue theorem, Theorem~\ref{tstokes501}, to obtain a 
Stokes version of the Friedlander eigenvalue theorem.

\section{An abstract Friedlander comparison theorem} \label{Sfried2}

We recall the basic notions of form methods as developed in 
\cite{AE2} and \cite{AEKS}, which extend the classical form methods 
of Kato~\cite{Kat1} and Lions~\cite{Lio4}.
Since we only use self-adjoint graphs in this paper, we restrict our attention to 
symmetric forms. 

Let $V$, $K$ be Hilbert spaces and $\gota \colon V \times V \to \Ci$ 
be a sesquilinear form. 
The form $\gota$ is called {\bf continuous} if there is an $M > 0$ such that
$|\gota(u,v)| \leq M \, \|u\|_V \, \|v\|_V$ for all $u,v \in V$.
The form $\gota$ is called {\bf positive} if $\gota(u,u) \geq 0$
for all $u \in V$ and it is called {\bf symmetric} if 
$\gota(v,u) = \overline{\gota(u,v)}$ for all $u,v \in V$.
Let $j \colon V \to K$ be a continuous linear map.
In Lions \cite{Lio4}, the map $j$ is an inclusion map.
We say that $\gota$ is {\bf $j$-elliptic} if there are 
$\omega,\mu > 0$ such that 
\[
\RRe \gota(u,u) + \omega \, \|j(u)\|_K^2
\geq \mu \, \|u\|_V^2
\]
for all $u \in V$.
The {\bf graph associated with $(\gota,j)$} is the subspace
\[
\{ (j(u),f) \in K \times K : u \in V \mbox{ and } 
   \gota(u,v) = (f,j(v))_K \mbox{ for all } v \in V \} 
\]
in $K \times K$.

There are a number of sufficient conditions to conclude that the 
graph associated with $(\gota,j)$ is a self-adjoint graph if 
$\gota$ is continuous and symmetric.
For a summary of the terminology in self-adjoint graphs we refer
to \cite{AEKS} Section~3 and \cite{BE} Section~2.
In this paper we use the following two sufficient conditions.

\begin{prop} \label{pfried240}
Let $V$, $K$ be Hilbert spaces and $\gota \colon V \times V \to \Ci$ 
be a continuous symmetric sesquilinear form. 
Let $j \colon V \to K$ be a continuous linear map.
Suppose at least one of the following conditions is valid.
\begin{tabel} 
\item \label{pfried240-1}
The form $\gota$ is $j$-elliptic.
\item \label{pfried240-2}
There exists a Hilbert space $H$ and a compact linear map 
$i \colon V \to H$ such that $\gota$ is $i$-elliptic.
\end{tabel}
Then the graph $A$ associated with $(\gota,j)$ is a self-adjoint 
graph which is lower bounded.

Moreover, if $j$ is compact, then $A$ has compact resolvent 
and in particular a discrete spectrum.
\end{prop}
\begin{proof}
`\ref{pfried240-1}'.
This follows from \cite{AE2} Theorem~2.1 applied to the Hilbert space~$\overline{j(V)}$.

`\ref{pfried240-2}'.
See \cite{AEKS} Theorems~4.5 and 4.15.

For the last statement, see \cite{AEKS} Proposition~4.8.
\end{proof}

Now we are able to formulate the abstract eigenvalue theorem.
We denote by $A^\circ$ the single valued part of a self-adjoint graph $A$.
For an immediate example we refer the reader to the beginning of the 
proof of Theorem~\ref{tfried202}.

\begin{thm} \label{tstokes501}
Let $V$, $H$ and $K$ be Hilbert spaces.
Suppose that $V$ is embedded in $H$ and that the inclusion map $i \colon V \to H$ is compact.
Let $j \colon V \to K$ be a compact linear map.
Let $\gota \colon V \times V \to \Ci$ be a positive symmetric continuous $i$-elliptic 
sesquilinear form.
Let $V_D = \ker j$.
Let $A^N$ be the self-adjoint operator in $\overline V \subset H$ 
associated with $\gota$ and let 
$A^D$ be the self-adjoint operator in $\overline{V_D} \subset H$ 
associated with $\gota|_{V_D \times V_D}$, where the 
closures are in $H$.
Further, for all $\lambda \in \Ri$ define $\gotb_\lambda \colon V \times V \to \Ci$
by 
\[
\gotb_\lambda(u,v) 
= \gota(u,v) - \lambda \, (u,v)_H
.  \]
Let $\cn_\lambda$ be the self-adjoint graph associated with $(\gotb_\lambda,j)$.
Suppose 
\begin{tabelR}
\item \label{tstokes501-1}
the operator $A^N$ has no eigenvector in $V_D$ and 
\item \label{tstokes501-2}
$\displaystyle 
\dim \spann \{ \varphi \in D(\cn_\lambda) : (\cn_\lambda^\circ \varphi,\varphi)_K = 0 \} 
= \infty
$
for all $\lambda \in (0,\infty)$.
\end{tabelR}
Let $\lambda_1^N \leq \lambda_2^N \leq \ldots$ and 
$\lambda_1^D \leq \lambda_2^D \leq \ldots$ be the eigenvalues of $A^N$ and $A^D$
repeated with multiplicity, respectively.
Then 
\[
\lambda_{n+1}^N < \lambda_n^D
\]
for all $n \in \Ni$.
\end{thm}
\begin{proof}
Since $\gota$ is $i$-elliptic there are $\omega,\delta > 0$ such that 
\begin{equation}
\gota(u,u) + \omega \, \|u\|_H^2 \geq \delta \, \|u\|_V^2
\label{etstokes501;1}
\end{equation}
for all $u \in V$.
Let $\widehat A^D$ be the self-adjoint graph in $H$ associated with 
$(\gota|_{V_D \times V_D}, i|_{V_D})$.
Then $A^D = (\widehat A^D)^\circ$, the singular part of $\widehat A^D$
and $\widehat A^D$ is positive.
Moreover, $\widehat A^D$ has compact resolvent since $i$ is compact.

For all $\mu \in \Ri$ define the form $\gota_\mu \colon V \times V \to \Ci$ by
\[
\gota_\mu(u,v) = \gota(u,v) - \mu \, (j(u), j(v))_K
.  \]
Then $\gota_\mu$ is $i$-elliptic.
Let $A_\mu$ be the self-adjoint graph associated with $(\gota_\mu,i)$.
Then $A_\mu$ has compact resolvent by Proposition~\ref{pfried240}.
Let $\lambda_1(\mu) \leq \lambda_2(\mu) \leq \ldots$ be the eigenvalues
of $A_\mu$, repeated with multiplicity.
Then $\lambda_n \colon \Ri \to \Ri$ is a function for all $n \in \Ni$.
Since $\mu \mapsto \gota_\mu(u)$ is a decreasing function for all $u \in V$,
it follows from the mini-max theorem
that $\lambda_n$ is a decreasing function for all $n \in \Ni$.

\begin{lemma} \label{lstokes502}
$\lim_{\mu \to - \infty} (I + A_\mu)^{-1} = (I + \widehat A^D)^{-1}$ in $\cl(H)$.
\end{lemma}
\begin{proof}
Let $\mu_1,\mu_2,\ldots \in (-\infty,0)$ and suppose that 
$\lim_{n \to \infty} \mu_n = - \infty$.
We shall prove that $\lim_{n \to \infty} \|(I + A_{\mu_n})^{-1} - (I + \widehat A^D)^{-1}\|_{\cl(H)} = 0$.
Suppose not. 
Then there are $\varepsilon > 0$ and $f_1,f_2,\ldots \in H$ such that 
\begin{equation}
\|(I + A_{\mu_n})^{-1} f_n - (I + \widehat A^D)^{-1} f_n\|_H 
> \varepsilon \, \|f_n\|_H
\label{elstokes502;2}
\end{equation}
for all $n \in \Ni$.
Without loss of generality we may assume that $\|f_n\|_H = 1$ for all $n \in \Ni$.
Passing to a subsequence if necessary, there exists an $f \in H$ 
such that $\lim f_n = f$ weakly in $H$.
Then 
$\lim (I + \widehat A^D)^{-1} f_n = (I + \widehat A^D)^{-1} f$ in $H$
since $(I + \widehat A^D)^{-1}$ is compact.
For all $n \in \Ni$ set $u_n = (I + A_{\mu_n})^{-1} f_n$.

Let $n \in \Ni$.
Then 
\begin{equation}
\gota(u_n,v) - \mu_n \, (j(u_n), j(v))_K + (u_n, v)_H
= (f_n, v)_H
\label{elstokes502;1}
\end{equation}
for all $v \in V$.
Choosing $v = u_n$ gives
\[
\gota(u_n, u_n) + |\mu_n| \, \|j(u_n)\|_K^2 + \|u_n\|_H^2
= \RRe (f_n, u_n)_H
\leq \|f_n\|_H \, \|u_n\|_H
= \|u_n\|_H
.  \]
So $\|u_n\|_H \leq 1$ and then also 
$\gota(u_n, u_n) \leq 1$ and $|\mu_n| \, \|j(u_n)\|_K^2 \leq 1$.
Hence $\lim j(u_n) = 0$ in~$K$.
Moreover, $\delta \, \|u_n\|_V^2 \leq (1 + \omega)$
for all $n \in \Ni$ by (\ref{etstokes501;1}).
Therefore the sequence $(u_n)_{n \in \Ni}$ is bounded in $V$.
Passing to a subsequence if necessary, there exists a $u \in V$ 
such that $\lim u_n = u$ weakly in $V$.
Then $\lim u_n = u$ strongly in $H$ since the embedding $i$ is compact.
Similarly $j(u) = \lim j(u_n)$ in $K$.
But $\lim j(u_n) = 0$ in $K$.
So $j(u) = 0$ and $u \in V_D$.

Let $v \in V_D$.
Taking the limit $n \to \infty$ in (\ref{elstokes502;1}) gives
$\gota(u,v) + (u,v)_H = (f,v)_H$.
So $u \in D(\widehat A^D)$ and $u = (I + \widehat A^D)^{-1} f$.
We proved that $\lim u_n = u$ in $H$.
Hence 
\[
\lim_{n \to \infty}
(I + A_{\mu_n})^{-1} f_n - (I + \widehat A^D)^{-1} f_n
= (I + \widehat A^D)^{-1} f - (I + \widehat A^D)^{-1} f 
= 0
\]
in $H$.
This contradicts (\ref{elstokes502;2}).
\end{proof}

\begin{prop} \label{pstokes503}
If $n \in \Ni$, then $\lim_{\mu \to - \infty} \lambda_n(\mu) = \lambda_n^D$.
\end{prop}
\begin{proof}
Let $n \in \Ni$.
Let $\varepsilon > 0$.
By Lemma~\ref{lstokes502} there exists an $M \in \Ni$ such that 
\[
( (I + \widehat A^D)^{-1} f,f)_H - \varepsilon \, \|f\|_H^2
\leq ( (I + A_\mu)^{-1} f,f)_H
\leq ( (I + \widehat A^D)^{-1} f,f)_H + \varepsilon \, \|f\|_H^2
\]
for all $f \in H$ and $\mu \in (-\infty,-M]$.
Note that the max-min theorem gives
\[
( 1 + \lambda_n(\mu))^{-1}
= \max_{ \scriptstyle W \subset H \atop
         \scriptstyle \dim W = n}
  \min_{ \scriptstyle f \in W  \atop
         \scriptstyle \|f\|_H = 1}
   ( (I + A_\mu)^{-1} f,f)_H
\]
and a similar expression is valid for $\widehat A^D$.
So 
$( 1 + \lambda_n^D)^{-1} - \varepsilon 
\leq ( 1 + \lambda_n(\mu))^{-1} 
\leq ( 1 + \lambda_n^D)^{-1} + \varepsilon$
for all $\mu \in (-\infty,-M]$ and the proposition follows.
\end{proof}

The next lemma is a version of the Birman--Schwinger principle.

\begin{lemma} \label{lstokes504}
Let $\lambda,\mu \in \Ri$.
Then $\lambda \in \sigma(A_\mu)$ if and only if $\mu \in \sigma(\cn_\lambda)$.
\end{lemma}
\begin{proof}
`$\Rightarrow$'.
Suppose $\lambda \in \sigma(A_\mu)$.
Then there exists a $u \in D(A_\mu)$ such that $u \neq 0$ and 
$A_\mu u = \lambda \, u$.
Then 
\begin{equation}
\gota(u,v) - \mu \, (j(u), j(v))_K = (\lambda \, u, v)_H
\label{elstokes504;1}
\end{equation}
for all $v \in V$.
Hence $\gotb_\lambda(u,v) = \gota(u,v) - \lambda \, (u, v)_H = ( \mu \, j(u), j(v))_K$
for all $v \in V$.
Therefore $j(u) \in D(\cn_\lambda)$ and $(j(u), \mu \, j(u)) \in \cn_\lambda$.
It remains to show that $j(u) \neq 0$.
Suppose that $j(u) = 0$.
Then it follows from (\ref{elstokes504;1}) that $\gota(u,v) = (\lambda \, u, v)_H$
for all $v \in V$.
So $u \in D(A^N)$ and $A^N u = \lambda \, u$.
Hence $u$ is an eigenvector for $A^N$ and also $u \in V_D$.
By Assumption~\ref{tstokes501-1} this is not possible.
So $j(u) \neq 0$ and $\mu \in \sigma(\cn_\lambda)$.

`$\Leftarrow$'.
The proof is similar, even easier.
\end{proof}

\begin{lemma} \label{lstokes505}
Let $n \in \Ni$.
Then $\lambda_n$ is strictly decreasing.
\end{lemma}
\begin{proof}
We already know that $\lambda_n$ is decreasing.
Let $\mu_1,\mu_2 \in \Ri$ with $\mu_1 < \mu_2$.
Suppose that $\lambda_n(\mu_1) = \lambda_n(\mu_2)$.
Write $\lambda = \lambda_n(\mu_1)$.
Let $\mu \in [\mu_1,\mu_2]$.
Since $\lambda_n$ is decreasing, it follows that $\lambda_n(\mu) = \lambda$.
So $\lambda \in \sigma(A_\mu)$.
Hence $\mu \in \sigma(\cn_\lambda)$ by Lemma~\ref{lstokes504}.
So $[\mu_1,\mu_2] \subset \sigma(\cn_\lambda)$.
This is a contradiction since $\cn_\lambda$ has a discrete spectrum
by Proposition~\ref{pfried240}.
\end{proof}

\begin{lemma} \label{lstokes506}
Let $\lambda > 0$. 
Then $\sigma(\cn_\lambda) \cap (-\infty,0) \neq \emptyset$.
\end{lemma}
\begin{proof}
Let $\cn_\lambda^\circ$ be the single valued part of $\cn_\lambda$.
Suppose $\sigma(\cn_\lambda) \cap (-\infty,0) = \emptyset$.
Then $\cn_\lambda^\circ$ is a positive self-adjoint operator.

Let $\varphi \in D(\cn_\lambda)$ and suppose that 
$(\cn_\lambda^\circ \varphi,\varphi)_K = 0$.
Let $\psi \in D(\cn_\lambda^\circ)$.
Then 
\[
|(\cn_\lambda^\circ \varphi, \psi)_K|
\leq \|(\cn_\lambda^\circ)^{1/2} \varphi\|_K \, \|(\cn_\lambda^\circ)^{1/2} \psi\|_K
= (\cn_\lambda^\circ \varphi,\varphi)_K^{1/2} \, \|(\cn_\lambda^\circ)^{1/2} \psi\|_K
= 0
.  \]
Hence $\cn_\lambda^\circ \varphi = 0$ and $\varphi \in \ker(\cn_\lambda^\circ)$.
By Assumption~\ref{tstokes501-2} 
\[
\dim \spann \{ \varphi \in D(\cn_\lambda) : (\cn_\lambda^\circ \varphi,\varphi)_K = 0 \} 
= \infty
.  \]
So $\dim \ker(\cn_\lambda^\circ) = \infty$.
But the operator $\cn_\lambda^\circ$ has compact resolvent by Proposition~\ref{pfried240}.
This is a contradiction.
\end{proof}

Now we complete the proof of Theorem~\ref{tstokes501}.
Let $n \in \Ni$.
Choose $\lambda = \lambda_n^D$.
Since $\lambda_n$ is strictly decreasing by Lemma~\ref{lstokes505}, it follows from 
Proposition~\ref{pstokes503} that 
\[
0 \leq \lambda_n^N = \lambda_n(0) < \lambda_n^D = \lambda
.  \]
Let $\mu = \min \sigma(\cn_\lambda)$.
Then $\mu < 0$ by Lemma~\ref{lstokes506}.
Moreover, $\lambda \in \sigma(A_\mu)$ by Lemma~\ref{lstokes504}.
Hence there exists a $j \in \Ni$ such that $\lambda_j(\mu) = \lambda$.
Using again Lemma~\ref{lstokes505} and Proposition~\ref{pstokes503}
it follows that $\lambda_n(\mu) < \lambda_n^D = \lambda = \lambda_j(\mu)$.
Therefore $j \geq n+1$.
Then 
\[
\lambda_{n+1}^N 
= \lambda_{n+1}(0) 
< \lambda_{n+1}(\mu) 
\leq \lambda_j(\mu) 
= \lambda 
= \lambda_n^D
,  \]
where the strict inequality follows from Lemma~\ref{lstokes505}.
\end{proof}

The first main example is the classical version of Friedlander's theorem.

\begin{thm} \label{tfried202}
Let $\Omega \subset \Ri^d$ be a bounded open set with Lipschitz boundary
and $d \geq 2$.
Let $\Delta^N$ and $\Delta^D$ be the Neumann and Dirichlet Laplacians on $\Omega$,
respectively.
Further let $\lambda_1^N \leq \lambda_2^N \leq \ldots$ and 
$\lambda_1^D \leq \lambda_2^D \leq \ldots$ be the eigenvalues of $-\Delta^N$ and $-\Delta^D$
repeated with multiplicity, respectively.
Then 
\[
\lambda_{n+1}^N < \lambda_n^D
\]
for all $n \in \Ni$.
\end{thm}
\begin{proof}
Choose $V = H^1(\Omega)$, $H = L_2(\Omega)$ and $K = L_2(\partial \Omega)$.
Let $i \colon V \to H$ be the inclusion map and $j = \Tr \colon V \to K$. 
Then $i$ and $j$ are compact.
Moreover, $\ker j = H^1_0(\Omega)$.
Define $\gota \colon V \times V \to \Ci$ by 
\[
\gota(u,v) 
= \int_\Omega \nabla u \cdot \overline{ \nabla v }
.  \]
Then $\gota$ is a positive symmetric continuous $i$-elliptic sesquilinear form.
Further, $- \Delta^N$ is the self-adjoint operator associated with $\gota$
and $- \Delta^D$ is the self-adjoint operator associated with 
$\gota|_{H^1_0(\Omega) \times H^1_0(\Omega)} = \gota|_{\ker j \times \ker j}$.

It remains to verify Conditions~\ref{tstokes501-1} and \ref{tstokes501-2}
of Theorem~\ref{tstokes501}.

`\ref{tstokes501-1}'.
Let $u \in D(A^N) \cap V_D$ and $\lambda \in \Ri$.
Suppose that $A^N u = \lambda \, u$.
Then $u \in H^1_0(\Omega)$.
Let $\tilde u \in H^1(\Ri^d)$ be the extension of $u$ by zero.
Then 
\[
\int_{\Ri^d} \nabla \tilde u \cdot \overline{\nabla v}
= \int_\Omega \nabla u \cdot \overline{\nabla (v|_\Omega)}
= (\lambda \, u, v|_\Omega)_{L_2(\Omega)}
= (\lambda \, \tilde u, v)_{L_2(\Ri^d)}
\]
for all $v \in H^1(\Ri^d)$.
So $- \widetilde \Delta^N \tilde u = \lambda \, \tilde u$, where 
$\widetilde \Delta^N$ is the Laplacian on $\Ri^d$.
Since $\tilde u$ vanishes on a non-empty open set, it follows from the 
unique continuation property that $\tilde u = 0$.
(See for example \cite{RS4} Theorem~XIII.57.)
Then also $u = 0$.
So $A^N$ has no eigenvector in~$V_D$.

`\ref{tstokes501-2}'.
Let $\lambda \in (0,\infty)$.
Let $\omega \in \Ri^d$ and suppose that $|\omega|^2 = \lambda$.
Define $u \in H^2(\Omega)$ by $u(x) = e^{i \omega \cdot x}$.
Then $- \Delta u = \lambda u$ as distribution and 
\begin{eqnarray*}
\gotb_\lambda(u,v)
& = & \int_\Omega \nabla u \cdot \overline{ \nabla v }
     - \lambda \int_\Omega u \, \overline v  \\
& = & \int_\Omega \nabla u \cdot \overline{ \nabla v }
     + \int_\Omega (\Delta u) \, \overline v  \\
& = & \int_\Omega \divv( \overline v \, \nabla u)  \\
& = & \int_{\partial \Omega} \nu \cdot \Tr(\overline v \, \nabla u)
= \int_{\partial \Omega} i (\omega \cdot \nu) \, (\Tr u) \, \overline{\Tr v}
= (i (\omega \cdot \nu) \, \Tr u, j(v))_K
\end{eqnarray*}
for all $v \in C_b^\infty(\Omega)$, where we used the divergence theorem
(see \cite{AltEnglish} Theorem~A8.8).
So by density (see for example \cite{MazED2} Theorem~1.1.6/2)
one deduces that 
\[
\gotb_\lambda(u,v)
= (i (\omega \cdot \nu) \, \Tr u, j(v))_K
\]
for all $v \in H^1(\Omega)$.
Hence
$\Tr u \in D(\cn_\lambda)$ and $(\Tr u , i (\omega \cdot \nu) \, \Tr u) \in \cn_\lambda$.
Moreover, 
\[
(\cn_\lambda^\circ j(u), j(u))_{L_2(\Omega, \Ci^d)}
= \gotb_\lambda(u,u)
= \int_{\partial \Omega} i (\omega \cdot \nu) \, |\Tr u|^2
= i \int_{\partial \Omega} (\omega \cdot \nu)
= i \int_\Omega \divv \omega
= 0
.  \]
Therefore 
\[
\Tr u
\in \{ \varphi \in D(\cn_\lambda) : (\cn_\lambda^\circ \varphi,\varphi)_{L_2(\partial \Omega)} = 0 \} 
.  \]
It is then obvious that 
\[
\dim \spann \{ \varphi \in D(\cn_\lambda) : (\cn_\lambda^\circ \varphi,\varphi)_{L_2(\partial \Omega)} = 0 \} 
= \infty
.  \]
This completes the proof of the theorem.
\end{proof}

\section{The Stokes operators} \label{Sfried3}

In this section we describe the Dirichlet and the Neumann version 
of the Stokes operator on a Lipschitz domain with the aid of 
sesquilinear forms.
We then apply Theorem~\ref{tstokes501} (abstract eigenvalues comparison theorem) to obtain
a Stokes version of Friedlander's eigenvalues theorem.
For independent interest we present a description of the 
Stokes type Dirichlet-to-Neumann operator.

Let $\Omega \subset \Ri^d$ be a bounded open connected set with Lipschitz boundary.
We assume that $d \geq 2$.
We denote by $\nu$ the outward normalised normal.
Let $H = L_2(\Omega,\Ci^d)$,
\[
V = \{ u \in H^1(\Omega,\Ci^d) : \divv u = 0 \}
\]
and $K = L_2(\partial \Omega, \Ci^d)$.
Here and below $\divv$ is the divergence operator acting on distributions.
Then the inclusion $i \colon V \to H$ is compact.
Let $j = \Tr \colon V \to K$ be the trace operator.
Then also $j$ is compact.
Define the form $\gota \colon V \times V \to \Ci$ by 
\[
\gota(u,v)
= \int_\Omega \nabla u \cdot \overline{ \nabla v}
.  \]
Then $\gota$ is continuous and $i$-elliptic.
Let $A^N$ be the self-adjoint operator in $\overline V \subset H$ 
associated with $\gota$ and let 
$A^D$ be the self-adjoint operator in $\overline{V_D} \subset H$ 
associated with $\gota|_{V_D \times V_D}$, where the 
closures are in $H$ and $V_D = \ker j$.
We first characterise $A^D$ and $A^N$.

Define
\[
C^\infty_{c,\sigma}(\Omega)
= \{ u \in C_c^\infty(\Omega,\Ci^d) : \divv u = 0 \} 
,  \]
the space of {\bf solenoidal test functions},
and let $L_{2,\sigma}(\Omega)$ be the closure of $C^\infty_{c,\sigma}(\Omega)$
in $L_2(\Omega,\Ci^d)$.

\begin{prop} \label{pfried302}
\mbox{}
\begin{tabel}
\item \label{pfried302-1}
The closure of the space $V_D$ in $L_2(\Omega,\Ci^d)$ is $L_{2,\sigma}(\Omega)$.
\item \label{pfried302-3}
$V_D^\perp = L_{2,\sigma}(\Omega)^\perp 
= \{ \nabla \pi : \pi \in H^1(\Omega) \} $,
where the orthogonal complement is in $L_2(\Omega,\Ci^d)$.
\item \label{pfried302-2}
Let $u,f \in L_{2,\sigma}(\Omega)$.
Then $u \in D(A^D)$ and $A^D u = f$ if and only if $u \in H^1_0(\Omega,\Ci^d)$,
$\divv u = 0$ and there exists a $\pi \in L_2(\Omega)$ such that 
$f = - \Delta u + \nabla \pi$ in $H^{-1}(\Omega,\Ci^d)$.
\end{tabel}
\end{prop}
\begin{proof}
`\ref{pfried302-1}'.
This follows from \cite{Sohr} Lemma~II.2.2.3.

`\ref{pfried302-3}'.
The first equality is obvious.
The second one follows from \cite{Temam} Theorem~1.1.4.

`\ref{pfried302-2}'.
Clearly $V_D = \{ w \in H^1_0(\Omega,\Ci^d) : \divv w = 0 \} $.
`$\Rightarrow$'.
Since $A^D$ is the operator associated with $\gota|_{V_D \times V_D}$
it follows that $u \in V_D$ and $\gota(u,v) = (f,v)_{L_2(\Omega,\Ci^d)}$
for all $v \in V_D$.
So $u \in H^1_0(\Omega,\Ci^d)$ and $\divv u = 0$.
Consider $(f + \Delta u) \in H^{-1}(\Omega,\Ci^d)$.
If $v \in H^1_0(\Omega,\Ci^d)$ and $\divv v = 0$, then 
\[
(f + \Delta u)(v)
= (f,v)_{L_2(\Omega,\Ci^d)} 
    + \langle \Delta u, v \rangle_{H^{-1}(\Omega,\Ci^d) \times H^1_0(\Omega,\Ci^d)}
= (f,v)_{L_2(\Omega,\Ci^d)} - \gota(u,v)
= 0
.  \]
Since $\Omega$ has a Lipschitz boundary, it follows from 
\cite{Sohr} Lemma~II.2.1.1b that there exists a $\pi \in L_2(\Omega)$
such that $f + \Delta u = \nabla \pi$ in $H^{-1}(\Omega,\Ci^d)$.

`$\Leftarrow$'.
Obviously $u \in V_D$.
If $v \in V_D$, then $v \in H^1_0(\Omega,\Ci^d)$ and $\divv v = 0$.
So 
\begin{eqnarray*}
\gota(u,v)
& = & - \langle \Delta u,v \rangle_{H^{-1}(\Omega,\Ci^d) \times H^1_0(\Omega,\Ci^d)}
= \langle f - \nabla \pi,v \rangle_{H^{-1}(\Omega,\Ci^d) \times H^1_0(\Omega,\Ci^d)}  \\
& = & (f,v)_{L_2(\Omega,\Ci^d)} + (\pi, \divv v)_{L_2(\Omega)}
= (f,v)_{L_2(\Omega,\Ci^d)}
\end{eqnarray*}
and the proposition follows.
\end{proof}

We next consider the operator $A^N$.
In order to obtain a description as in Proposition~\ref{pfried302}\ref{pfried302-2}
for $A^D$, we first need a kind of normal derivative.
If $\Omega$ is smooth, $u \in H^2(\Omega,\Ci^d)$, $\pi \in H^1(\Omega)$
and $f \in L_2(\Omega,\Ci^d)$ 
with $- \Delta u + \nabla \pi = f$ in $H^{-1}(\Omega,\Ci^d)$, 
then 
\begin{eqnarray*}
\lefteqn{
\langle \partial_\nu u - (\Tr \pi) \, \nu, \Tr \Phi \rangle_{H^{-1/2} \times H^{1/2}}
} \hspace*{20mm}  \\*
& = & \int_\Omega \nabla u \cdot \overline{\nabla \Phi} 
   + \langle \Delta u, \Phi \rangle_{H^{-1} \times H^1}
   - \langle \nabla \pi, \Phi \rangle_{H^{-1} \times H^1}
   - \int_\Omega \pi \, \overline{\divv \Phi}  \\
& = & \int_\Omega \nabla u \cdot \overline{\nabla \Phi} 
   - \int_\Omega \pi \, \overline{\divv \Phi}
   - \int_\Omega f \cdot \overline{\Phi}
\end{eqnarray*}
for all $\Phi \in H^1(\Omega,\Ci^d)$.
We next define a weak version of this normal derivative.

\begin{lemma} \label{lstokes403}
Let $u \in H^1(\Omega,\Ci^d)$, $\pi \in L_2(\Omega)$ and $f \in L_2(\Omega,\Ci^d)$.
Suppose that $- \Delta u + \nabla \pi = f$ in $H^{-1}(\Omega,\Ci^d)$.
Then there exists a unique $F \in H^{-1/2}(\partial \Omega, \Ci^d)$
such that 
\[
\langle F, \varphi \rangle_{H^{-1/2}(\partial \Omega, \Ci^d) \times H^{1/2}(\partial \Omega, \Ci^d)}
= \int_\Omega \nabla u \cdot \overline{\nabla \Phi} 
   - \int_\Omega \pi \, \overline{\divv \Phi}
   - \int_\Omega f \cdot \overline{\Phi}
\]
for all $\varphi \in H^{1/2}(\partial \Omega, \Ci^d)$ and
$\Phi \in H^1(\Omega,\Ci^d)$ with $\Tr \Phi = \varphi$.
\end{lemma}
\begin{proof}
Let $\Phi \in H^1(\Omega,\Ci^d)$ and suppose that $\Tr \Phi = 0$.
Then 
\begin{eqnarray*}
\int_\Omega \nabla u \cdot \overline{\nabla \Phi} 
   - \int_\Omega \pi \, \overline{\divv \Phi}
   - \int_\Omega f \cdot \overline{\Phi}
& = & \langle - \Delta u, \Phi \rangle_{H^{-1} \times H^1}
   + \langle \nabla \pi, \Phi \rangle_{H^{-1} \times H^1} 
   - \int_\Omega f \cdot \overline{\Phi}   \\
& = & \langle - \Delta u + \nabla \pi - f, \Phi \rangle_{H^{-1}(\Omega,\Ci^d) \times H^1(\Omega,\Ci^d)}
= 0 .
\end{eqnarray*}
Hence there exists a unique $F \colon H^{1/2}(\partial \Omega, \Ci^d) \to \Ci$
such that 
\[
F(\varphi) 
= \int_\Omega \nabla u \cdot \overline{\nabla \Phi} 
   - \int_\Omega \pi \, \overline{\divv \Phi}
   - \int_\Omega f \cdot \overline{\Phi}
\]
for all $\varphi \in H^{1/2}(\partial \Omega, \Ci^d)$ and
$\Phi \in H^1(\Omega,\Ci^d)$ with $\Tr \Phi = \varphi$.
Obviously $F$ is anti-linear and 
\[
|F(\Tr \Phi)| \leq (\|u\|_{H^1(\Omega,\Ci^d)} 
                     + \|\pi\|_{L_2(\Omega)}
                     + \|f\|_{L_2(\Omega,\Ci^d)} ) \|\Phi\|_{H^1(\Omega,\Ci^d)}
\]
for all $\Phi \in H^1(\Omega,\Ci^d)$.
Since $\Tr \colon H^1(\Omega,\Ci^d) \to H^{1/2}(\partial \Omega, \Ci^d)$
has a continuous right-inverse, the lemma follows.
\end{proof}

We denote the element $F \in H^{-1/2}(\partial \Omega, \Ci^d)$ in Lemma~\ref{lstokes403}
by $\partial_\nu(u,\pi)$.
So 
\[
\langle \partial_\nu(u,\pi), \Tr \Phi \rangle_{H^{-1/2}(\partial \Omega, \Ci^d) \times H^{1/2}(\partial \Omega, \Ci^d)}
= \int_\Omega \nabla u \cdot \overline{\nabla \Phi} 
   - \int_\Omega \pi \, \overline{\divv \Phi}
   - \int_\Omega f \cdot \overline{\Phi}
\]
for all $u \in H^1(\Omega,\Ci^d)$, $\pi \in L_2(\Omega)$, $f\in L_2(\Omega,\Ci^d)$
and $\Phi \in H^1(\Omega,\Ci^d)$ with 
$- \Delta u + \nabla \pi = f$ in $H^{-1}(\Omega,\Ci^d)$.

Let $u \in H^1(\Omega,\Ci^d)$, $\pi \in L_2(\Omega)$, $f \in L_2(\Omega,\Ci^d)$ 
and $c \in \Ci$ with 
$- \Delta u + \nabla \pi = f$ in $H^{-1}(\Omega,\Ci^d)$.
Then $- \Delta u + \nabla (\pi + c \, \one_\Omega) = f$ in $H^{-1}(\Omega,\Ci^d)$
and 
\begin{equation}
\partial_\nu(u,\pi + c \, \one_\Omega) = \partial_\nu(u,\pi) - c \, \nu
\label{efriedS4;1}
\end{equation}
by a straightforward calculation.

We also need to characterise the range $\Tr V$ of the form domain under the 
trace map.
Note that $\Tr \colon H^1(\Omega,\Ci^d) \to H^{1/2}(\partial \Omega, \Ci^d)$.
If $u \in V$, then the divergence theorem gives
\[
0 = \int_\Omega \divv u
= \int_{\partial \Omega} \nu \cdot \Tr u
.  \]
Define 
\[
L_{2,0}(\partial \Omega,\Ci^d) 
= \{ \varphi \in L_2(\partial \Omega, \Ci^d) : 
        \int_{\partial \Omega} \varphi \cdot \nu = 0 \}
.  \]
Then $\Tr(V) \subset H^{1/2}(\partial \Omega, \Ci^d) \cap L_{2,0}(\partial \Omega,\Ci^d)$.
We next show that one actually has an equality.

\begin{lemma} \label{lstokes401}
$\Tr(V) = H^{1/2}(\partial \Omega, \Ci^d) \cap L_{2,0}(\partial \Omega,\Ci^d)$.
\end{lemma}
\begin{proof}
We only have to show `$\supset$'.
Let $\varphi \in H^{1/2}(\partial \Omega, \Ci^d) \cap L_{2,0}(\partial \Omega,\Ci^d)$.
There exists a $u \in H^1(\Omega,\Ci^d)$ such that 
$\Tr u = \varphi$.
Then $\int_\Omega \divv u = \int_{\partial \Omega} \nu \cdot \Tr u
= \int_{\partial \Omega} \nu \cdot \varphi = 0$.
Hence by \cite{Sohr} Lemma~II.2.1.1a there exists a $w \in H^1_0(\Omega,\Ci^d)$
such that $\divv w = \divv u$.
Set $v = u - w$.
Then $v \in H^1(\Omega,\Ci^d)$ and $\divv v = 0$.
So $v \in V$.
Moreover, $\Tr v = \Tr u = \varphi$.
\end{proof}

Although we do not need the next density lemma, we state it
for completeness.

\begin{lemma} \label{lstokes402}
The range $\Tr(V)$ is dense in $L_{2,0}(\partial \Omega,\Ci^d)$.
\end{lemma}
\begin{proof}
Let $\varphi \in L_{2,0}(\partial \Omega,\Ci^d)$.
Since $H^{1/2}(\partial \Omega, \Ci^d)$ is dense in $L_2(\partial \Omega,\Ci^d)$, 
there is a sequence $(\varphi_n)_{n \in \Ni}$ in $H^{1/2}(\partial \Omega, \Ci^d)$
such that 
$\lim \varphi_n = \varphi$ in $L_2(\partial \Omega,\Ci^d)$.
Then 
$\lim \int_{\partial \Omega} \varphi_n \cdot \nu
= \int_{\partial \Omega} \varphi \cdot \nu = 0$.
Note that $\nu \in L_\infty(\partial \Omega,\Ci^d) \subset L_2(\partial \Omega,\Ci^d)$.
Hence as above there is a sequence $(\nu_n)_{n \in \Ni}$ 
in $H^{1/2}(\partial \Omega, \Ci^d)$ such that 
$\lim \nu_n = \nu$ in $L_2(\partial \Omega,\Ci^d)$.
Then $\lim \int_{\partial \Omega} \nu_n \cdot \nu = \sigma(\partial \Omega) \neq 0$,
so we may assume that $\int_{\partial \Omega} \nu_n \cdot \nu \neq 0$ 
for all $n \in \Ni$.
For all $n \in \Ni$ define 
$\tilde \varphi_n = \varphi_n - c_n \, \nu_n$,
where $c_n = \Big( \int_{\partial \Omega} \nu_n \cdot \nu \Big)^{-1} 
                   \int_{\partial \Omega} \varphi_n \cdot \nu$.
Then $\tilde \varphi_n \in H^{1/2}(\partial \Omega, \Ci^d) \cap L_{2,0}(\partial \Omega,\Ci^d) = \Tr(V)$
by Lemma~\ref{lstokes401}.
Moreover, $\lim \tilde \varphi_n = \varphi$ in $L_2(\partial \Omega,\Ci^d)$.
\end{proof}

Now we are able to characterise the operator $A^N$.

\begin{prop} \label{pfried303}
\mbox{}
\begin{tabel}
\item \label{pfried303-1}
The closure of the space $V$ in $L_2(\Omega,\Ci^d)$ is 
$ \{ u \in L_2(\Omega,\Ci^d) : \divv u = 0 \} $.
\item \label{pfried303-3}
$V^\perp = \{ \nabla \pi : \pi \in H^1_0(\Omega) \} $,
where the orthogonal complement is in $L_2(\Omega,\Ci^d)$.
\item \label{pfried303-2}
Let $u,f \in \overline V^{(H)}$.
Then $u \in D(A^N)$ and $A^N u = f$ if and only if $u \in H^1(\Omega,\Ci^d)$,
$\divv u = 0$ and there exists a $\pi \in L_2(\Omega)$ such that 
$f = - \Delta u + \nabla \pi$ in $H^{-1}(\Omega,\Ci^d)$ 
and $\partial_\nu(u,\pi) = 0$.
\end{tabel}
\end{prop}
\begin{proof}
`\ref{pfried303-1}'. 
See \cite{MMW} Lemma~2.1.

`\ref{pfried303-3}'. 
By definition of $V$ we obtain that $V^\perp$ is the closure of 
$ \{ \nabla \pi : \pi \in C_c^\infty(\Omega) \} $ in $L_2(\Omega, \Ci^d)$.
Obviously the latter is the closure of the space 
$ \{ \nabla \pi : \pi \in H^1_0(\Omega) \} $ in $L_2(\Omega, \Ci^d)$.
By \cite{MazED2} Corollary~1.1.11 and the Dirichlet Poincar\'e inequality the space 
$ \{ \nabla \pi : \pi \in H^1_0(\Omega) \} $ in closed in $L_2(\Omega, \Ci^d)$.

`\ref{pfried303-2}'. 
`$\Rightarrow$'.
By definition $u \in V$ and $\gota(u,v) = (f,v)_{L_2(\Omega,\Ci^d)}$
for all $v \in V$.
So $u \in H^1(\Omega,\Ci^d)$ and $\divv v = 0$.
Consider $(f + \Delta u) \in H^{-1}(\Omega,\Ci^d)$.
If $v \in H^1_0(\Omega,\Ci^d)$ and $\divv v = 0$, then 
\[
(f + \Delta u)(v)
= (f,v)_{L_2(\Omega,\Ci^d)} 
    + \langle \Delta u, v \rangle_{H^{-1}(\Omega,\Ci^d) \times H^1_0(\Omega,\Ci^d)}
= (f,v)_{L_2(\Omega,\Ci^d)} - \gota(u,v)
= 0
.  \]
Since $\Omega$ has a Lipschitz boundary, it follows from 
\cite{Sohr} Lemma~II.2.1.1b that there exists a $\pi \in L_2(\Omega)$
such that $f + \Delta u = \nabla \pi$ in $H^{-1}(\Omega,\Ci^d)$.

If $v \in V$, then 
\begin{eqnarray*}
\langle \partial_\nu(u,\pi), \Tr v \rangle_{H^{-1/2}(\partial \Omega, \Ci^d) \times H^{1/2}(\partial \Omega, \Ci^d)}
& = & \int_\Omega \nabla u \cdot \overline{\nabla v} 
   - \int_\Omega \pi \, \overline{\divv v}
   - \int_\Omega f \cdot \overline v  \\
& = & \gota(u,v)
   - (f,v)_{L_2(\Omega,\Ci^d)}
= 0 .
\end{eqnarray*}
So by Lemma~\ref{lstokes401} one deduces that 
$\langle \partial_\nu(u,\pi), \varphi \rangle_{H^{-1/2}(\partial \Omega, \Ci^d) \times H^{1/2}(\partial \Omega, \Ci^d)}
= 0$
for all $\varphi \in H^{1/2}(\partial \Omega, \Ci^d) \cap L_{2,0}(\partial \Omega,\Ci^d)$.

Fix $\varphi_0 \in H^{1/2}(\partial \Omega, \Ci^d)$ such that 
$(\varphi_0,\nu)_{L_2(\partial \Omega, \Ci^d)} \neq 0$.
Let $c \in \Ci$ be such that 
\[
\langle \partial_\nu(u,\pi) - c \, \nu, 
     \varphi_0 \rangle_{H^{-1/2}(\partial \Omega, \Ci^d) \times H^{1/2}(\partial \Omega, \Ci^d)}
= 0
.  \]
Then it follows from (\ref{efriedS4;1}) that 
\[
\langle \partial_\nu(u,\pi + c \, \one_\Omega), 
      \varphi \rangle_{H^{-1/2}(\partial \Omega, \Ci^d) \times H^{1/2}(\partial \Omega, \Ci^d)}
= 0
\]
for all
 $\varphi \in \{ \varphi_0 \} \cup \Big( H^{1/2}(\partial \Omega, \Ci^d) \cap L_{2,0}(\partial \Omega,\Ci^d) \Big)$ 
and then by linearity for all 
\[
\varphi \in \spann \bigg( \{ \varphi_0 \} \cup \Big( H^{1/2}(\partial \Omega, \Ci^d) \cap L_{2,0}(\partial \Omega,\Ci^d) \Big) \bigg)
= H^{1/2}(\partial \Omega, \Ci^d)
.  \]
So $\partial_\nu(u,\pi + c \, \one_\Omega) = 0$.
Replacing $\pi$ by $\pi + c \, \one_\Omega$ completes the proof of the 
implication~`$\Rightarrow$'.

`$\Leftarrow$'.
Let $u \in H^1(\Omega,\Ci^d)$ and $f \in \overline V^{(H)}$.
Suppose that 
$\divv u = 0$ and there exists a $\pi \in L_2(\Omega)$ such that 
$f = - \Delta u + \nabla \pi$ and $\partial_\nu(u,\pi) = 0$.
Then $u \in V$ and 
\[
\int_\Omega \nabla u \cdot \overline{\nabla v} 
   - \int_\Omega \pi \, \overline{\divv v}
   - \int_\Omega f \cdot \overline v
= \langle \partial_\nu(u,\pi), \Tr v \rangle_{H^{-1/2}(\partial \Omega, \Ci^d) \times H^{1/2}(\partial \Omega, \Ci^d)}
= 0
\]
for all $v \in H^1(\Omega,\Ci^d)$.
Hence if $v \in V$, then $\divv v = 0$ and  
\[
\gota(u,v) = \int_\Omega \nabla u \cdot \overline{\nabla v} 
= \int_\Omega f \cdot \overline v
= (f,v)_{L_2(\Omega,\Ci^d)}
.  \]
So $u \in \dom(A^N)$ and $A^N u = f$.
\end{proof}

We next turn to a Stokes version of the Dirichlet-to-Neumann operator.
For all $\lambda \in \Ri$ define $\gotb_\lambda \colon V \times V \to \Ci$
by 
\[
\gotb_\lambda(u,v) 
= \gota(u,v) - \lambda \, (u,v)_H
.  \]
Let $\cn_\lambda$ be the self-adjoint graph in $L_2(\partial \Omega,\Ci^d)$
associated with $(\gotb_\lambda,j)$.
We call $\cn_\lambda$ the {\bf Stokes Dirichlet-to-Neumann graph}.
The singular part $\cn_\lambda^\circ$ is a self-adjoint operator in the 
Hilbert space $\overline{D(\cn_\lambda)}$, where the closure is in 
$L_2(\partial \Omega, \Ci^d)$.
Note that 
$\overline{D(\cn_\lambda)}
\subset \overline{j(V)}
= \overline{\Tr V}
= L_{2,0}(\partial \Omega,\Ci^d)$
by Lemma~\ref{lstokes402}.
We next characterise $\cn_\lambda^\circ$ in the case where
$\lambda$ is not an eigenvalue of $A^D$.

\begin{prop} \label{pstokes404}
Let $\lambda \in \Ri \setminus \sigma(A^D)$.
Let $\varphi,\psi \in L_{2,0}(\partial \Omega,\Ci^d)$.
Then the following are equivalent.
\begin{tabeleq}
\item \label{pstokes404-1}
$\varphi \in D(\cn_\lambda^\circ)$ and $\cn_\lambda^\circ \varphi = \psi$.
\item \label{pstokes404-2}
There exist $u \in H^1(\Omega,\Ci^d)$ and $\pi \in L_2(\Omega)$
such that
\begin{itemize}
\item
\quad $\varphi = \Tr u$,
\item
\quad $\divv u = 0$,  
\item
\quad $- \Delta u + \nabla \pi = \lambda \, u$ in $H^{-1}(\Omega,\Ci^d)$, and
\item
\quad $\psi = \partial_\nu(u,\pi)$.
\end{itemize}
\end{tabeleq}
\end{prop}
\begin{proof}
`\ref{pstokes404-1}$\Rightarrow$\ref{pstokes404-2}'.
By definition there exists a $u \in V$ such that $\Tr u = \varphi$ and 
$\gotb_\lambda(u,v) = (\psi,\Tr v)_{L_2(\partial \Omega,\Ci^d)}$ for all $v \in V$.
Consider $(-\Delta u - \lambda \, u) \in H^{-1}(\Omega,\Ci^d)$.
If $v \in H^1_0(\Omega,\Ci^d)$ and $\divv v = 0$, then 
\[
(-\Delta u - \lambda \, u)(v)
= \gotb_\lambda(u,v)
= (\psi,\Tr v)_{L_2(\partial \Omega,\Ci^d)}
= 0
.  \]
Since $\Omega$ has a Lipschitz boundary, it follows from 
\cite{Sohr} Lemma~II.2.1.1b that there exists a $\pi \in L_2(\Omega)$
such that $-\Delta u - \lambda \, u = - \nabla \pi$ in $H^{-1}(\Omega,\Ci^d)$.

Let $v \in V$.
Then 
\begin{eqnarray*}
(\psi,\Tr v)_{L_2(\partial \Omega,\Ci^d)}
& = & \gotb_\lambda(u,v)
= \int_\Omega \nabla u \cdot \overline{\nabla v} 
   - \lambda \int_\Omega u \cdot \overline v  \\
& = & \int_\Omega \nabla u \cdot \overline{\nabla v} 
   - \lambda \int_\Omega u \cdot \overline v 
   - \int_\Omega \pi \, \overline{\divv v}
= \langle \partial_\nu(u,\pi), \Tr v \rangle_{H^{-1/2} \times H^{1/2}}
.  
\end{eqnarray*}
So 
$\langle \partial_\nu(u,\pi) - \psi, \tau \rangle_{H^{-1/2} \times H^{1/2}} = 0$
for all $\tau \in \Tr(V) = H^{1/2}(\partial \Omega, \Ci^d) \cap L_{2,0}(\partial \Omega,\Ci^d)$
by Lemma~\ref{lstokes401}.
Now it follows as at the end of the proof of the implication `$\Rightarrow$' in 
the proof of Proposition~\ref{pfried303}\ref{pfried303-2} that there exists a 
$c \in \Ci$ such that $\partial_\nu(u,\pi + c \, \one_\Omega) = \psi$.

`\ref{pstokes404-2}$\Rightarrow$\ref{pstokes404-1}'.
Let $u$ and $\pi$ be as in \ref{pstokes404-2}.
Note that $u \in V$.
Let $v \in V$.
Then 
\begin{eqnarray*}
\gotb_\lambda(u,v)
& = & \int_\Omega \nabla u \cdot \overline{\nabla v} 
   - \lambda \int_\Omega u \cdot \overline v 
   - \int_\Omega \pi \, \overline{\divv v}  \\
& = & \langle \partial_\nu(u,\pi), \Tr v \rangle_{H^{-1/2} \times H^{1/2}}
= \langle \psi, \Tr v \rangle_{H^{-1/2} \times H^{1/2}}
= (\psi, \Tr v)_{L_2(\partial\Omega,\Ci^d)}
\end{eqnarray*}
and \ref{pstokes404-1} follows.
\end{proof}

Now we are able to state and prove the main theorem of this paper.
For an explicit description of the Neumann Stokes operator and 
Dirichlet Stokes operator we refer to Propositions~\ref{pfried303}\ref{pfried303-2} 
and~\ref{pfried302}\ref{pfried302-2}.

\begin{thm} \label{tfried305}
Let $\Omega \subset \Ri^d$ be a bounded open connected set with Lipschitz boundary.
Let $A^N$ and $A^D$ be the Neumann Stokes operator and Dirichlet Stokes operator on $\Omega$,
respectively.
Further let $\lambda_1^N \leq \lambda_2^N \leq \ldots$ and 
$\lambda_1^D \leq \lambda_2^D \leq \ldots$ be the eigenvalues of $A^N$ and~$A^D$
repeated with multiplicity, respectively.
Then 
\[
\lambda_{n+1}^N < \lambda_n^D
\]
for all $n \in \Ni$.
\end{thm}
\begin{proof}
We have to verify Conditions~\ref{tstokes501-1} and \ref{tstokes501-2}
of Theorem~\ref{tstokes501}.

`\ref{tstokes501-1}'.
Let $u \in D(A^N) \cap \ker j$ and suppose that $u$ is an 
eigenvector for $A^N$ with eigenvalue $\lambda \in \Ri$.
Then $u \in H^1_0(\Omega,\Ci^d)$ and $\divv u = 0$ in $\Omega$.
Let $\tilde u \in H^1(\Ri^d,\Ci^d)$ be the extension by zero of $u$.
Then $\nabla \tilde u = \widetilde{\nabla u}$, the extension by zero 
of $\nabla u$.
So $\divv \tilde u = 0$ in $\Ri^d$.
Let $v \in H^1(\Ri^d,\Ci^d)$ and suppose that $\divv v = 0$ in $\Ri^d$.
Then 
\[
\int_{\Ri^d} \nabla \tilde u \cdot \overline{ \nabla v}
= \int_\Omega \nabla u \cdot \overline{ \nabla (v|_\Omega)}
= \gota(u, v|_\Omega)
= (\lambda \, u, v|_\Omega)_{L_2(\Omega,\Ci^d)}
= (\lambda \, \tilde u, v)_{L_2(\Ri^d,\Ci^d)}
.  \]
So $\tilde u \in D(A_{\Ri^d})$, the Stokes operator on $\Ri^d$,
and $A_{\Ri^d} \tilde u = \lambda \, \tilde u$.
Then $- \Delta \tilde u = \lambda \, \tilde u$ by \cite{Sohr} Lemma~III.2.3.2.
But $\tilde u$ vanishes on a non-empty open set.
So $\tilde u = 0$ by the unique continuation property, see
for example \cite{RS4} Theorem~XIII.57.
Hence $u = 0$ and $u$ is not an eigenvector.

`\ref{tstokes501-2}'.
Let $\lambda > 0$.
Let $\omega \in \Ri^d$ and suppose that $|\omega|^2 = \lambda$.
Let $b \in \Ri^d$ be such that $|b| = 1$ and $b \cdot \omega = 0$.
Define $\tau \in H^2(\Omega)$ by $\tau(x) = e^{i \omega \cdot x}$.
Then $- \Delta \tau = |\omega|^2 \, \tau = \lambda \, \tau$.
Let $u = \tau \, b \in H^1(\Omega,\Ci^d)$.
Then $\divv u = i \tau \, (b \cdot \omega) = 0$, so $u \in V$.
Let $v \in V$.
Then 
\begin{eqnarray*}
\gotb_\lambda(u,v)
& = & \int_\Omega \nabla u \cdot \overline{ \nabla v} - \lambda \int_\Omega u \cdot \overline v  \\
& = & \int_\Omega \nabla \tau \cdot \overline{ \nabla (b \cdot v) }
    - \lambda \int_\Omega \tau \, \overline{(b \cdot v)} \\
& = & - \int_\Omega (\Delta \tau)  \, \overline{(b \cdot v)}
   + \int_{\partial \Omega} (\partial_\nu \tau) \, \overline{ \Tr (b \cdot v) }
    - \lambda \int_\Omega \tau \, \overline{(b \cdot v)}  \\
& = & \int_{\partial \Omega} (\partial_\nu \tau) \, b \cdot \overline{ \Tr v}
= \int_{\partial \Omega} i \, (\omega \cdot \nu) \, (\Tr \tau) \, b \cdot \overline{ \Tr v}
.
\end{eqnarray*}
Therefore $j(u) \in D(\cn_\lambda)$ and 
$(j(u),  i \, (\omega \cdot \nu) \, (\Tr \tau) \, b) \in \cn_\lambda$.
Moreover, 
\begin{eqnarray*}
(\cn_\lambda^\circ j(u), j(u))_{L_2(\Omega, \Ci^d)}
& = & \gotb_\lambda(u,u)  \\
& = & i \int_{\partial \Omega} (\omega \cdot \nu) \, |b|^2 \, |\Tr \tau|^2
= i \, |b|^2 \int_{\partial \Omega} (\omega \cdot \nu)
= i \, |b|^2 \int_\Omega \divv \omega
= 0
.
\end{eqnarray*}
Consequently
\[
\dim \spann \{ \varphi \in D(\cn_\lambda) : (\cn_\lambda^\circ \varphi,\varphi)_{L_2(\Omega, \Ci^d)} = 0 \} 
= \infty
\]
as required.
\end{proof}

We finish with a small extension of the previous theorem.

\begin{exam} \label{xstokes508}
Let $\Omega \subset \Ri^d$ be a bounded open connected set with Lipschitz boundary.
Let $H = L_2(\Omega,\Ci^d)$,
\[
V = \{ u \in H^1(\Omega,\Ci^d) : \divv u = 0 \}
\]
and $K = L_2(\partial \Omega, \Ci^d)$.
Then the inclusion $i \colon V \to H$ is compact.
Let $j = \Tr \colon V \to K$ be the trace operator.
Then $j$ is compact.
Fix $\alpha \in (-1,1]$.
Define the form $\gota \colon V \times V \to \Ci$ by 
\[
\gota(u,v)
= \int_\Omega \nabla u \cdot \overline{ \nabla v}
+ \alpha \int_\Omega (\nabla u)^{\rm T} \cdot \overline{ \nabla v}
,  \]
where the superscript ${\rm T}$ denotes the transpose.
So Theorem~\ref{tfried305} corresponds to the case~$\alpha = 0$.
Obviously $\gota$ is continuous and if $|\alpha| < 1$, then 
$\gota$ is $i$-elliptic.
If $\alpha = 1$, then
\[
\RRe \gota(u)
= \frac{1}{2} \int_\Omega |\nabla u + (\nabla u)^{\rm T}|^2
\]
for all $u \in V$.
Since $\Omega$ has a Lipschitz boundary, Korn's second inequality states that 
there exists a $C > 0$ 
such that 
\[
\|u\|_{H^1(\Omega,\Ci^d)}^2
\leq C \Big( \int_\Omega |\nabla u + (\nabla u)^{\rm T}|^2
             + \| u \|_{L_2(\Omega,\Ci^d)}^2 \Big)
\]
for all $u \in H^1(\Omega,\Ci^d)$.
For an easy proof, see \cite{Nitsche} Section~3.
Hence the form $\gota$ is also $i$-elliptic if $\alpha = 1$.
The associated operator $A^N$ is the Neumann Stokes operator studied in \cite{MMW}
and $A^D$ is the (Dirichlet) Stokes operator.
We next show that the conditions in Theorem~\ref{tstokes501} are valid.

`\ref{tstokes501-1}'.
Let $u \in D(A^N) \cap \ker j$ and suppose that $u$ is an 
eigenvector for $A^N$ with eigenvalue $\lambda \in \Ri$.
Then as before $u \in H^1_0(\Omega,\Ci^d)$ and $\divv u = 0$ in $\Omega$.
Let $\tilde u \in H^1(\Ri^d,\Ci^d)$ be the extension by zero of $u$.
Again $\nabla \tilde u = \widetilde{\nabla u}$, the extension by zero 
of $\nabla u$, and $\divv \tilde u = 0$ in $\Ri^d$.
Let $v \in H^1(\Ri^d,\Ci^d)$ and suppose that $\divv v = 0$ in $\Ri^d$.
Then 
\begin{eqnarray*}
\int_{\Ri^d} \Big( \nabla \tilde u + \alpha \, (\nabla \tilde u)^{\rm T} \Big) 
        \cdot \overline{ \nabla v}
& = & \int_\Omega \Big( \nabla u + \alpha \, (\nabla u)^{\rm T} \Big) \cdot \overline{ \nabla (v|_\Omega)}
= \gota(u, v|_\Omega)  \\
& = & (\lambda \, u, v|_\Omega)_{L_2(\Omega,\Ci^d)}
= (\lambda \, \tilde u, v)_{L_2(\Ri^d,\Ci^d)}
.  
\end{eqnarray*}
Since $\divv \tilde u = 0$, one deduces that 
\[
\int_{\Ri^d}  (\nabla \tilde u)^{\rm T} 
        \cdot \overline{ \nabla v}
= 0
.  \]
Therefore $\tilde u \in D(A_{\Ri^d})$, the Stokes operator on $\Ri^d$,
and $A_{\Ri^d} \tilde u = \lambda \, \tilde u$.
Now one can argue as in the proof of Theorem~\ref{tfried305} that $u = 0$
and hence $u$ is not an eigenvector.

`\ref{tstokes501-2}'.
Let $\lambda > 0$.
Let $\omega \in \Ri^d$ and suppose that $|\omega|^2 = \lambda$.
As before let $b \in \Ri^d$ and $\tau \in H^2(\Omega)$
be such that $|b| = 1$, $b \cdot \omega = 0$ and $\tau(x) = e^{i \omega \cdot x}$.
Then $\Delta \tau = - |\omega|^2 \, \tau = - \lambda \, \tau$.
Let $u = \tau \, b$.
Then $\divv u = i \tau \, (b \cdot \omega) = 0$, so $u \in V$.
Let $v \in V$.
Then 
\begin{eqnarray*}
\gotb_\lambda(u,v)
& = & \int_\Omega \nabla u \cdot \overline{ \nabla v} 
   + \alpha \int_\Omega (\nabla u)^{\rm T} \cdot \overline{ \nabla v}
   - \lambda \int_\Omega u \cdot \overline v  \\
& = & \int_\Omega \nabla \tau \cdot \overline{ \nabla (b \cdot v) }
   + \alpha \sum_{k,l=1}^d \int_\Omega i \, \omega_k \, \tau \, b_l \, \overline{\partial_l v_k}
    - \lambda \int_\Omega \tau \, \overline{(b \cdot v)} \\
& = & - \int_\Omega (\Delta \tau)  \, \overline{(b \cdot v)}
   + \int_{\partial \Omega} (\partial_\nu \tau) \, \overline{ \Tr (b \cdot v) }  \\*
& & \hspace*{10mm} {}
   + \alpha \sum_{k,l=1}^d \int_\Omega \omega_k \, \omega_l \, \tau \, b_l \, \overline{v_k}
   + i \, \alpha \sum_{k,l=1}^d \int_{\partial \Omega} \omega_k \, (\Tr \tau) \, \nu_l \, b_l \, \overline{\Tr v_k}
    - \lambda \int_\Omega \tau \, \overline{(b \cdot v)}  \\
& = & \int_{\partial \Omega} (\partial_\nu \tau) \, b \cdot \overline{ \Tr v } 
   + i \, \alpha \int_{\partial \Omega} (\Tr \tau) \, (\nu \cdot b) \, \omega \cdot \overline{\Tr v}
,
\end{eqnarray*}
where we used in the last step that $b \cdot \omega = 0$.
Consequently $j(u) \in D(\cn_\lambda)$ and 
$(j(u),  (\partial_\nu \tau) \, b + i \alpha \, (\Tr \tau) \, (\nu \cdot b) \, \omega) \in \cn_\lambda$.
Moreover, 
\[
(\cn_\lambda^\circ j(u), j(u))_{L_2(\Omega, \Ci^d)}
= \gotb_\lambda(u,u)  \\
=  i \int_{\partial \Omega} (\omega \cdot \nu) \, |b|^2 \, |\Tr \tau|^2
,  \]
where we used once more that $b \cdot \omega = 0$.
Then argue as in the proof of Theorem~\ref{tfried305}.
\end{exam}

\section*{Acknowledgement}

The authors wish to thank Sylvie Monniaux for helpful discussions.
The authors are most grateful for the hospitality and fruitful stay of
the first-named author at the University of Auckland and the second-named
author at the Aix-Marseille Universit\'e. 
This work is partly supported by the Aix-Marseille Universit\'e, 
an NZ-EU IRSES counterpart fund and the Marsden Fund Council from
Government funding, administered by the Royal Society of New Zealand and the EU
Marie Curie IRSES program, project ‘AOS’, No. 318910.

\small 

\noindent
{\sc C. Denis,
Centre de Math\'ematiques et Informatique (CMI) Technop\^{o}le Ch\^{a}teau-Gombert,
39, rue F. Joliot Curie,
13453 Marseille Cedex 13, 
France}  \\
{\em E-mail address}\/: {\bf clement.denis@univ-amu.fr}

\smallskip

\noindent
{\sc A.F.M. ter Elst,
Department of Mathematics,
University of Auckland,
Private bag 92019,
Auckland 1142,
New Zealand}  \\
{\em E-mail address}\/: {\bf terelst@math.auckland.ac.nz}

\mbox{}

\end{document}